\documentclass[12pt]{article}

\usepackage[margin=1.3in, top=1.3in, bottom=1.3in]{geometry}

\usepackage[utf8]{inputenc} 
\usepackage[T1]{fontenc}    
\usepackage{url}            
\usepackage{booktabs}       
\usepackage{amsfonts}       
\usepackage{nicefrac}       
\usepackage{microtype}      

\usepackage{amsmath,amssymb,amsthm,xcolor}
\usepackage{array}
\usepackage{float}

\newtheorem{theorem}{\textit{Theorem}}
\newtheorem{lemma}{\textit{Lemma}}

\newtheorem{proposition}{\textit{Proposition}}

\usepackage{graphicx}
\usepackage{subcaption}
\usepackage{mdframed}
\usepackage{lipsum}
\newmdtheoremenv{theo}{Theorem}
\usepackage{wrapfig}
\usepackage[hidelinks,hypertexnames=false]{hyperref}
\usepackage{tikz}
\usepackage{enumitem}
\usepackage{changepage}
\allowdisplaybreaks

\title{\LARGE Nonsmooth rank-one matrix factorization landscape}

\begin{document}

\author{\large C\'edric Josz\thanks{\url{cj2638@columbia.edu}, IEOR, Columbia University, New York. Research supported by DTRA grant 13-1-0021, DARPA grant Lagrange, NSF grant 2023032, and ONR grant N00014-21-1-2282.} \and Lexiao Lai\thanks{\url{ll3352@columbia.edu}, IEOR, Columbia University, New York.}}
\date{}

\maketitle

\begin{center}
    \textbf{Abstract}
    \end{center}
    \vspace*{-3mm}
 \begin{adjustwidth}{0.2in}{0.2in}
~~~~We provide the first positive result on the nonsmooth optimization landscape of robust principal component analysis, to the best of our knowledge. It is the object of several conjectures and remains mostly uncharted territory. We identify a necessary and sufficient condition for the absence of spurious local minima in the rank-one case. Our proof exploits the subdifferential regularity of the objective function in order to eliminate the existence quantifier from the first-order optimality condition known as Fermat's rule.
\end{adjustwidth} 
\vspace*{3mm}

\section{Introduction}
\label{intro}
Low-rank matrix factorization has received significant attention in the last decade, initiated by several seminal papers \cite{gross2011,candes2011,recht2011,candes2010power}. 
It has various applications in data science and machine learning, which include principal component analysis \cite{candes2011,bertsimas2020solving}, facial recognition \cite{candes2011}, video surveillance \cite{bouwmans2017decomposition,garcia2020background}, recommender systems \cite{koren2009} and natural language processing \cite{levy2014}. The number of survey papers \cite{nguyen2019,charisopoulos2021low,chi2019,chen2018harnessing,jain2017,moitra2018algorithmic} on the subject in the last three years is a testament to the amount of research it has spawned.

While the exact recovery of a low-rank matrix via convex optimization is well understood \cite{candes2011,chandrasekaran2011}, its non-convex counterpart 
\begin{equation}
\label{eq:non-convex}
    \inf_{(X,Y)\in \mathbb{R}^{m\times r} \times \mathbb{R}^{n \times r}} ~~~ \| XY^T - M \|_1 
\end{equation}
remains elusive, where $M \in \mathbb{R}^{m\times n}$ and $\|A\|_1 := \sum_{i=1}^m \sum_{j=1}^n |A_{ij}|$ for any $A\in \mathbb{R}^{m \times n}$. Note that minimizing the Frobenius norm squared instead yields approximate recovery and is better understood \cite{baldi1989neural,valavi2020landscape,du2018algorithmic}. For a general data matrix $M \in \mathbb{R}^{m \times n}$, finding a global minimum to \eqref{eq:non-convex} is known to be an NP-hard problem \cite[Theorem 3]{gillis2018complexity}.
However if $M$ is itself a low-rank matrix, then it is conjectured that the objective function is sharp \cite{burke1993weak} in a neighborhood of the global minima (see \cite[Equation (2.11)]{li2019incremental} and \cite[Conjecture 8.7]{charisopoulos2021low}). At best, this would imply convergence guarantees for local search algorithms when initialized in a neighborhood of the global minima. In order to prove convergence to a global minimum from any random initial point, as observed in \cite{joszneurips2018,li2019,fattahi2020,anderson2019}, it is necessary to analyze the landscape. We do so in the rank-one case and obtain the following theorem.
\begin{theorem}
\label{thm:landscape}
Assume that $\mathrm{rank}(M)\leqslant1$. Then the function defined from $\mathbb{R}^m \times \mathbb{R}^n$ to $\mathbb{R}$ by
\begin{equation}
\label{eq:f1}
    f(x,y) :=  \sum\limits_{i=1}^m \sum\limits_{j=1}^n |x_iy_j-M_{ij}| 
\end{equation}
has no spurious local minima if and only if none or all of the entries of $M$ are equal to zero.
\end{theorem}

Spurious local minima are defined as local minima that are not global minima. The analysis of the landscape of the nonconvex and nonsmooth objective function $f$ poses a significant challenge. We next explain why several standard tools in optimization fall short of providing a way forward. As a first approach, one could consider the following equivalent constrained optimization problem
\begin{subequations}
    \begin{gather}
        \min_{(x,y,z)\in \mathbb{R}^m \times \mathbb{R}^n \times \mathbb{R}^{m \times n}} ~~~~~~~ \sum_{i=1}^m \sum_{j=1}^n z_{ij} ~~~~~~~~~~~~~~~~~~~~~~~~~~~~~~~~~~~~~~~~~~~~~~~~~ \\ ~~ \text{subject to} ~~~ -z_{ij} \leqslant x_iy_j-M_{ij} \leqslant z_{ij}, ~ i=1,\hdots,m, ~ j = 1,\hdots,n.
    \end{gather}
\end{subequations}
The advantage of such a formulation is, of course, the fact that it is smooth. One may then invoke the Karush-Kuhn-Tucker conditions. However, it is not clear how one should establish constraint qualification. Observe that when $M = (0,1)^T$, the gradients of the active constraints are linearly dependent at the origin, which is a saddle point. Besides, the Karush-Kuhn-Tucker conditions comprise an exponential numbers of cases as a function of the dimensions $m$ and $n$, and there is seemingly no way to avoid treating each case separately. 

As a second approach, one could try to approximate the nonconvex and nonsmooth objective function \eqref{eq:f1} by a smooth function. This viewpoint underlies many approaches for solving nonsmooth optimization problems, notably proximal methods \cite{beck2017first} and Nesterov's smoothing technique \cite{nesterov2005smooth}. Along these lines, the first author and co-authors recently established that the uniform limit of a sequence of functions which are devoid of spurious local minima is itself devoid of spurious strict local minima \cite[Proposition 2.7]{joszneurips2018}. A spurious strict local minimum is a strict local minimum that is not a global minimum.
In our setting, a natural candidate of smooth approximations is given by the $\ell_p$-norm, where $p>1$:
\begin{equation}
\label{eq:fp}
    f_p(x,y) :=  \sum\limits_{i=1}^m \sum\limits_{j=1}^n |x_iy_j-M_{ij}|^p.
\end{equation}
By letting $p$ converge to 1 from above, $f_p$ converges uniformly to $f$ on any compact set. However, this approach fails to deliver any meaningful results for the problem at hand. Indeed, we know full well that the objective function \eqref{eq:f1} is devoid of spurious strict local minima. This follows from the invariance $f(\theta x , \theta^{-1}y) = f(x,y)$ for all $\theta \in \mathbb{R}\setminus \{0\}$, whereby any point different from the origin, itself a saddle point, has neighbors with identical function values.

A third approach, and the one that we will follow in this paper, is to use a generalized Fermat rule \cite[2.3.2 Proposition]{clarke1990} which holds for locally Lipschitz functions. It states that if $(x,y) \in \mathbb{R}^m \times \mathbb{R}^n$ is a local minimum of $f$, then $0 \in \partial f(x,y)$ where $\partial f$ is the Clarke subdifferential
\cite[pp. 25-27]{clarke1990}. We will refer to a point satisfying this set inclusion as a \textit{critical} point. The Clarke subdifferential of $f$ is defined for all $(x,y) \in \mathbb{R}^m \times \mathbb{R}^n$ by
\begin{equation}
    \partial f(x,y) := \{ (s,t) \in \mathbb{R}^m \times \mathbb{R}^n ~|~ f^\circ(x,y;h,k) \geqslant s^Th + t^Tk, ~~~ \forall (h,k)\in \mathbb{R}^m \times\mathbb{R}^n \} 
\end{equation}
where one uses the generalized directional derivative
\begin{equation}
    f^\circ(x,y;h,k) := \limsup_{\scriptsize\begin{array}{c} (\bar{x},\bar{y})\rightarrow (x,y) \\
    t \searrow 0
    \end{array}
    } \frac{f(\bar{x}+th,\bar{y}+tk)-f(\bar{x},\bar{y})}{t}.
\end{equation}
By virtue of \cite[2.3.10 Chain Rule II]{clarke1990}, we have the simpler form
\begin{equation}
\label{eq:subdif}
\partial f(x,y) = \left\{ \left. \begin{pmatrix} \Lambda y \\ \Lambda^T x \end{pmatrix} ~\right|~ \Lambda \in \mathrm{sign}(xy^T-M) \right\} 
\end{equation}
where 
\begin{equation}
\mathrm{sign}(t) :=
\left\{
\begin{array}{cl}
-1 & \text{if} ~ t < 0, \\
\big[-1,1\big] & \text{if} ~ t = 0, \\
\hphantom{-}1 & \text{if} ~ t > 0.
\end{array}
\right.
\end{equation}
Above, $[-1,1]$ stands for the interval in $\mathbb{R}$ that includes $-1$ and $1$, and the sign function applies to matrices term by term. It follows from \eqref{eq:subdif} that a point is critical if and only if
\begin{equation}
\label{eq:crit}
    \exists \Lambda \in \mathrm{sign}(xy^T-M): ~~~ \Lambda y = 0 ~~~ \text{and} ~~~ \Lambda^T x = 0.
\end{equation}
In order to study the variation of $f$ in the vicinity of its critical points, we seek to eliminate the quantifier $\exists$ from \eqref{eq:crit}. The Tarski-Seidenberg theorem \cite{tarski1951decision,seidenberg1954new} guarantees that this is possible for any $M\in \mathbb{R}^{m\times n}$, as long as $m$ and $n$ are fixed. One can thus recover the set of critical points when $M = (0,1)^T$ using commercial software, as can be seen in Figure \ref{fig:tarski}. However, no such result is guaranteed for any sizes $m$ and $n$. In order to overcome this, we analyze step functions that arise in the partial Clarke subdifferentials of $f$ in Lemma \ref{lemma:step}. (The partial Clarke subdifferential $\partial_{x_1} f(x,y)$ is, by definition \cite[p. 48]{clarke1990}, the Clarke subdifferential of $f(\cdot,x_2,\hdots,x_m,y)$ at $x_1$.)
We are thus able to eliminate the quantifier from \eqref{eq:crit} with Lemma \ref{lemma:eliminate}. The new description of the critical points is much more amenable to analysis. It yields a full characterization of the landscape of $f$ in Proposition \ref{prop:classify}, from which Theorem \ref{thm:landscape} is deduced.

\begin{figure}[ht]
    \centering
   \includegraphics[width=1\linewidth]{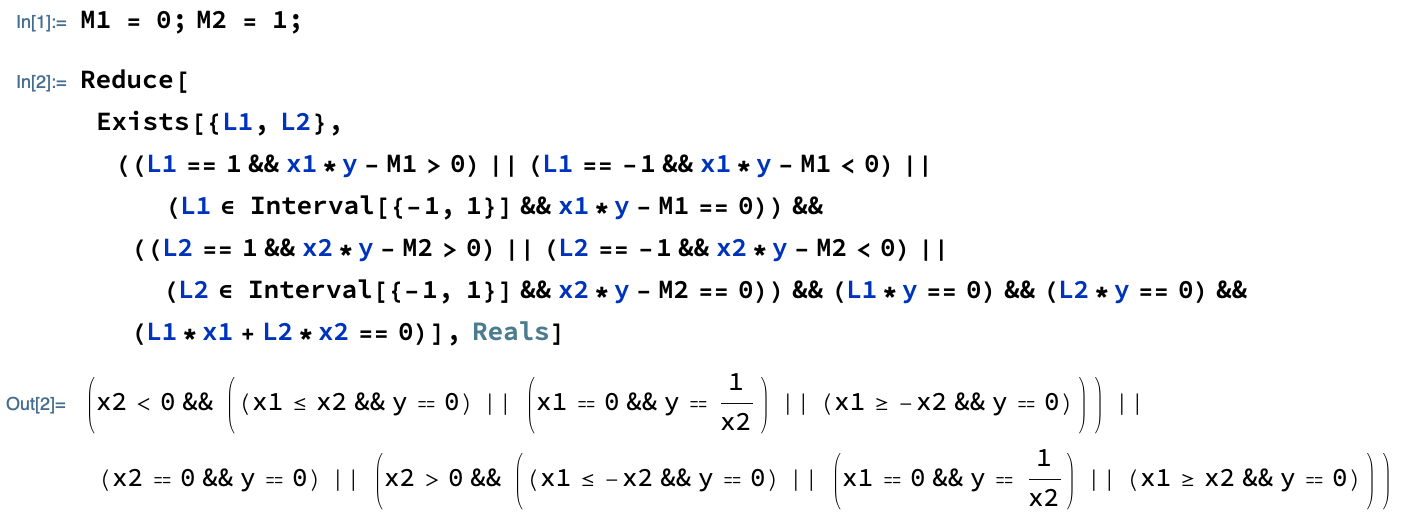}
  \caption{Elimination of quantifier with Wolfram Mathematica 12.}
  \label{fig:tarski}
\end{figure}

The reason why we can rely on partial Clarke subdifferentials to analyze the landscape is because the objective function $f$ is subdifferentially regular \cite[2.3.4 Definition]{clarke1990}. By definition, this means that its generalized directional derivative agrees with the classical directional derivative, that is to say, we have 
\begin{equation}
    \limsup_{\scriptsize\begin{array}{c} (\bar{x},\bar{y})\rightarrow (x,y) \\
    t \searrow 0
    \end{array}
    } \frac{f(\bar{x}+th,\bar{y}+tk)-f(\bar{x},\bar{y})}{t} ~ = ~ \lim_{
    t \searrow 0
    } \frac{f(x+th,y+tk)-f(x,y)}{t}
\end{equation}
for all $(x,y)\in \mathbb{R}^m \times \mathbb{R}^n$ and $(h,k) \in \mathbb{R}^m \times \mathbb{R}^n$, and the limit on the right hand side exists. As shown by Clarke \cite[2.5.2 Example]{clarke1990}, one should not take this property for granted. The function in Figure \ref{fig:irregular} defined from $\mathbb{R}^2$ to $\mathbb{R}$ by $\varphi(x_1,x_2) := \max \{ \min \{x_1,-x_2\} , x_2-x_1\}$ is not subdifferentially regular, despite being continuous and semi-algebraic, just like $f$, and hence belonging to the class of tame functions \cite{ioffe2009invitation}. As a result, its partial Clarke subdifferentials at the origin, which is a critical point, are completely decorrelated from its Clarke subdifferential: $\partial_{x_1} \varphi(0,0) \times \partial_{x_2} \varphi(0,0) \not\subset \partial \varphi(0,0) \not\subset \partial_{x_1} \varphi(0,0) \times \partial_{x_2} \varphi(0,0)$, where $\not\subset$ means ``is not a subset of'', as can be seen in Figure \ref{subdifs}. The partial Clarke subdifferentials are hence of no use to analyze the landscape of $\varphi$. In contrast, since $f$ is subdifferentially regular, we have 
\begin{equation}
\label{necessary_critical}
    \partial f(x,y) \subset \partial_{x_1}f(x,y) \times \hdots \times \partial_{x_m}f(x,y) \times \partial_{y_1}f(x,y) \times \hdots \times \partial_{y_n}f(x,y)
\end{equation}
where $\subset$ means ``is a subset of''. The inclusion is strict when
\begin{equation}
x = \begin{pmatrix}
\hphantom{-}1 \\
-1
\end{pmatrix},~~~ y = \begin{pmatrix}
1 \\ 
1
\end{pmatrix},~~~\text{and}~~~
    M = 
    \begin{pmatrix}
    \hphantom{-}2 & \hphantom{-}1 \\
    -1 & -1/2
    \end{pmatrix}
\end{equation}
since $0 \notin \partial f(x,y)$ yet $0 \in \partial_{x_1}f(x,y) \times  \partial_{x_2}f(x,y) \times \partial_{y_1}f(x,y) \times \partial_{y_2}f(x,y)$. It thus yields a necessary but not sufficient condition for being critical. The proof of Lemma \ref{lemma:step} makes extensive use of this necessary condition but ultimately invokes the critical condition to conclude.


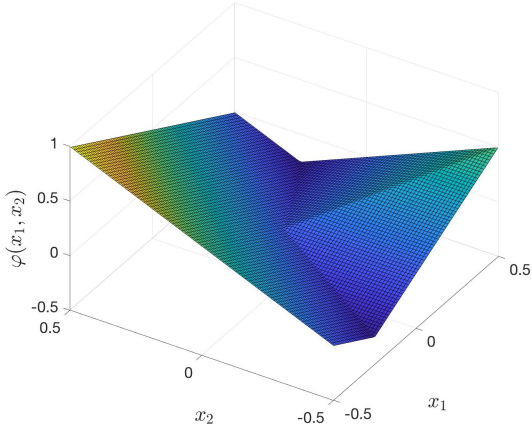
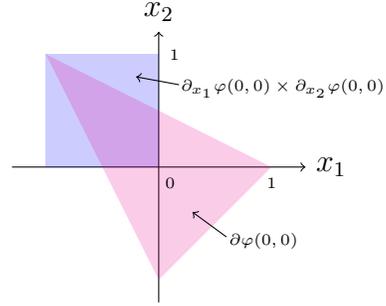
\begin{figure}[ht!]
\centering
\begin{subfigure}{.49\textwidth}
  \centering
  \includegraphics[width=1\linewidth]{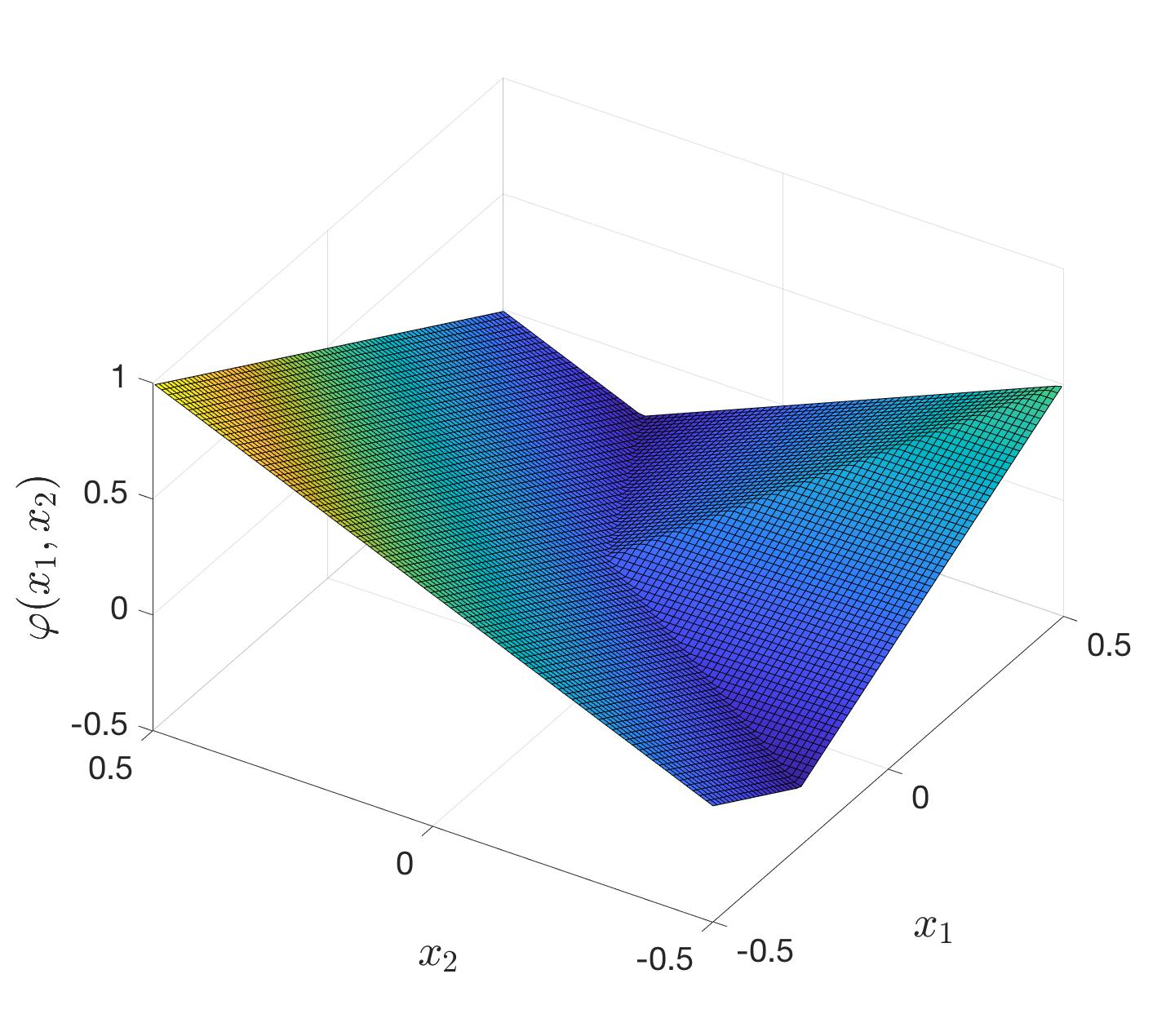}
  \caption{Graph around the origin.}
  \label{fig:irregular}
\end{subfigure}
\begin{subfigure}{.49\textwidth}
  \centering
   \vspace*{7.6mm}
  \begin{tikzpicture}[scale=1.5]
\draw[->] (-1.3,0)--(1.3,0) node[right]{$x_1$};
\draw[->] (0,-1.2)--(0,1.2) node[above]{$x_2$};
\fill[blue,opacity=0.2] (0,0) rectangle (-1,1);
\fill[magenta,opacity=0.2] (-1,1)--(1,0)--(0,-1);
\fill (1.1,.7)  node {\tiny $\partial_{x_1} \varphi(0,0) \times \partial_{x_2} \varphi(0,0)$};
\draw[->] (.18,.73) -- (-.2,.8);
\fill (.93,-.65)  node {\tiny $\partial \varphi(0,0)$};
\draw[->] (.6,-.62) -- (.3,-.4);
\fill (1,0)  node[below]{\tiny 1};
\fill (0,1)  node[right]{\tiny 1};
\fill (.1,0)  node[below]{\tiny 0};
\end{tikzpicture}
\vspace*{4mm}
  \caption{Subdifferentials at the origin.}
  \label{subdifs}
\end{subfigure}
\caption{Continuous semi-algebraic function that is not subdifferentially regular.
}
\end{figure}

\section{Proof of Theorem \ref{thm:landscape}}
\label{sec:landscape}
It will be convenient in our analysis to consider a factorization of $M$ whose rank we assume to be less than or equal to one. From now on, let $u\in \mathbb{R}^m$ and $v \in \mathbb{R}^n$ denote vectors such that $M=uv^T$. In order to introduce Lemma \ref{lemma:step}, observe that
\begin{subequations}
\begin{align}
    f(x,y) &= \sum\limits_{i=1}^m \sum\limits_{j=1}^n |x_iy_j-u_iv_j|\\
    & = \sum\limits_{u_i\neq 0} \sum\limits_{j=1}^n |x_iy_j-u_iv_j| + \sum\limits_{u_i=0} \sum\limits_{j=1}^n |x_iy_j| \\
    & = \sum\limits_{u_i\neq 0} |u_i| \sum\limits_{j=1}^n |y_j(x_i/u_i)-v_j| + \sum\limits_{j=1}^n |y_j| \sum\limits_{u_i=0}  |x_i|
\end{align}
\end{subequations}
where we use the convention that a sum over an index set which is empty is equal to zero. Similarly,
\begin{subequations}
\begin{align}
    f(x,y) & = \sum\limits_{j=1}^n \sum\limits_{i=1}^m |x_iy_j-u_iv_j|\\
    & = \sum\limits_{v_j\neq 0} \sum\limits_{i=1}^m |x_iy_j-u_iv_j| + \sum\limits_{v_j=0} \sum\limits_{i=1}^m |x_iy_j| \\
    & = \sum\limits_{v_j\neq 0} |v_j| \sum\limits_{i=1}^m |x_i(y_j/v_j)-u_i| + \sum\limits_{i=1}^m |x_i| \sum\limits_{v_j=0}  |y_j|.
\end{align}
\end{subequations}
As a result,
\begin{subequations} 
    \begin{align}
    \label{subeq:intro_a} f(x,y) & =  \sum\limits_{u_i\neq 0} |u_i| \alpha(x_i/u_i) + \|y\|_1 \sum\limits_{u_i= 0} |x_i| , \\
    f(x,y) & =  \sum\limits_{v_j\neq 0} |v_j| \beta(y_j/v_j) + \|x\|_1\sum\limits_{v_j=0} |y_j|,
    \label{subeq:intro_b}  
    \end{align}
\end{subequations}
where the functions $\alpha$ and $\beta$ are defined from $\mathbb{R}$ to $\mathbb{R}$ by
\begin{equation}
\label{eq:alpha_beta}
    \alpha(t)  := \sum\limits_{j=1}^n |y_j t-v_j| ~~~\text{and}~~~
    \beta(t)  := \sum\limits_{i=1}^m |x_i t-u_i|
\end{equation}
and $\|\cdot\|_1$ is the $\ell_1$-norm. The partial Clarke subdifferentials \cite[p. 48]{clarke1990} of $f$ are
\begin{subequations}
\begin{align}
    \partial_{x_i} f(x,y) = \left\{
    \begin{array}{cc}
        \mathrm{sign}(u_i) \partial \alpha (x_i/u_i)  & \text{if}~ u_i\neq 0, \\[1mm]
        \mathrm{sign}(x_i) \|y\|_1  & \text{if}~ u_i = 0, 
    \end{array}
    \right. \label{subdif_x} \\[3mm]
    \partial_{y_j} f(x,y) = \left\{
    \begin{array}{cc}
        \mathrm{sign}(v_j) \partial \beta (y_j/v_j)  & \text{if}~ v_j\neq 0, \\[1mm]
        \mathrm{sign}(y_j) \|x\|_1  & \text{if}~ v_j = 0,
    \end{array}
    \right. \label{subdif_y}
\end{align}
\end{subequations}
where 
\begin{equation}
\label{eq:step_functions}
    \partial \alpha(t) = \sum_{j=1}^n  \mathrm{sign}(y_jt-v_j)y_j ~~~\text{and}~~~ \partial \beta(t) = \sum\limits_{i=1}^m \mathrm{sign}(x_it-u_i)x_i.
\end{equation}
Since $\alpha$ and $\beta$ are convex piecewise affine functions, their subdifferentials $\partial \alpha$ and $\partial \beta$ are non-decreasing step functions, as can be seen in Figure \ref{fig:step}. From the expressions of $\partial \alpha(t)$ and $\partial \beta(t)$ in \eqref{eq:step_functions}, it follows that the jumps between the steps of $\partial \alpha$ occur at $v_j/y_j$ for all index $j$ such that $y_j\neq 0$, while those of $\partial \beta$ occur at $u_i/x_i$ such that $x_i \neq 0$. Observe that 0 is a root of $\partial \alpha$ and $\partial \beta$ in Figure \ref{fig:step}. This is true whenever $(x,y)$ is a critical point of $f$ and it is not a global minimum of $f$. Such is the object of Lemma \ref{lemma:step} below.

\begin{figure}[ht!]
\centering
\begin{subfigure}{.49\textwidth}
  \centering
  \begin{tikzpicture}[scale=0.75]
\draw[help lines, color=gray!200, dashed] (-3.3,-2.3) grid (3.3,2.3);
\draw[->] (-3.3,0)--(3.4,0) node[right]{$t$};
\draw[->] (0,-2.3)--(0,2.4) node[above]{$\partial \alpha(t)$};
\draw[blue,thick] (-3,-2)--(-1.9813,-2) ;
\draw[blue,thick] (-2,-2)--(-2,0) ;
\draw[blue,thick] (-2.0187,0)--(1.0187,0) ;
\draw[blue,thick] (1,0)--(1,2) ;
\draw[blue,thick] (0.9813,2)--(3,2) ;
\draw (-3,-2.3) node[below]{\tiny{-3}};
\draw (-2,-2.3) node[below]{\tiny{-2}};
\draw (-1,-2.3) node[below]{\tiny{-1}};
\draw (0,-2.3) node[below]{\tiny{0}};
\draw (1,-2.3) node[below]{\tiny{1}};
\draw (2,-2.3) node[below]{\tiny{2}};
\draw (3,-2.3) node[below]{\tiny{3}};
\draw (-3.3,-2) node[left]{\tiny{-2}};
\draw (-3.3,-1) node[left]{\tiny{-1}};
\draw (-3.3,0) node[left]{\tiny{0}};
\draw (-3.3,1) node[left]{\tiny{1}};
\draw (-3.3,2) node[left]{\tiny{2}};
\end{tikzpicture}
  \caption{Subdifferential of $\alpha$}
\end{subfigure}
\begin{subfigure}{.49\textwidth}
  \centering
  \begin{tikzpicture}[scale=0.75]
\draw[help lines, color=gray!200, dashed] (-3.3,-2.3) grid (3.3,2.3);
\draw[->] (-3.3,0)--(3.4,0) node[right]{$t$};
\draw[->] (0,-2.3)--(0,2.4) node[above]{$\partial \beta(t)$};
\draw[blue,thick] (-3,-2)--(-1.9813,-2) ;
\draw[blue,thick] (-2,-2)--(-2,-1.33) ;
\draw[blue,thick] (-2.0187,-1.33)--(-.9813,-1.33) ;
\draw[blue,thick] (-1,-1.33)--(-1,0) ;
\draw[blue,thick] (-1.0187,0)--(1.0187,0) ;
\draw[blue,thick] (1,0)--(1,1.33) ;
\draw[blue,thick] (0.9813,1.33)--(2.0187,1.33) ;
\draw[blue,thick] (2,1.33)--(2,2) ;
\draw[blue,thick] (1.9813,2)--(3,2) ;
\draw (-3,-2.3) node[below]{\tiny{-3}};
\draw (-2,-2.3) node[below]{\tiny{-2}};
\draw (-1,-2.3) node[below]{\tiny{-1}};
\draw (0,-2.3) node[below]{\tiny{0}};
\draw (1,-2.3) node[below]{\tiny{1}};
\draw (2,-2.3) node[below]{\tiny{2}};
\draw (3,-2.3) node[below]{\tiny{3}};
\draw (-3.3,-2) node[left]{\tiny{-6}};
\draw (-3.3,-1) node[left]{\tiny{-3}};
\draw (-3.3,0) node[left]{\tiny{0}};
\draw (-3.3,1) node[left]{\tiny{3}};
\draw (-3.3,2) node[left]{\tiny{6}};
\end{tikzpicture}
  \caption{Subdifferential of $\beta$}
\end{subfigure}
\caption{Step functions at the critical point $x = (2,-1,-1,1,-1)^T$, $y = (-1,-1/2,-1/2)^T$ where $u = (-2,-1,2,1,-2)^T$, $v = (-1,1,1)^T$.
}
\label{fig:step}
\end{figure}
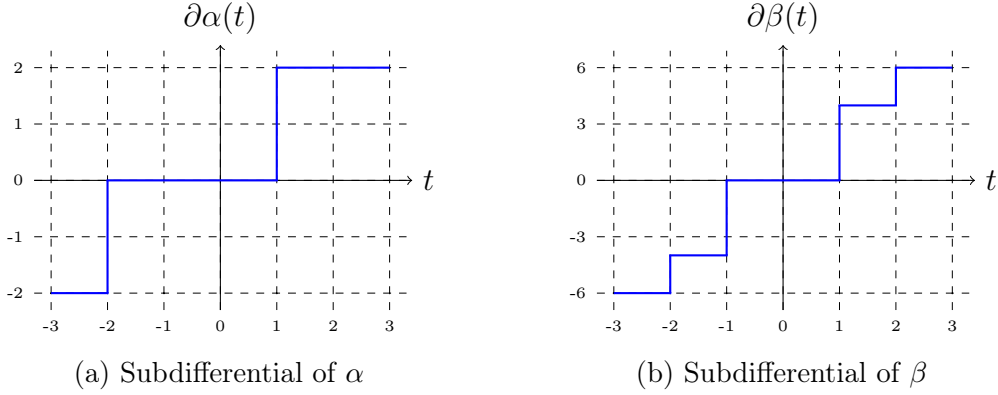
\begin{lemma}
\label{lemma:step}
$0 \in \partial f(x,y) ~~~ \Longrightarrow ~~~ f(x,y) = 0 ~~\text{or}~~ 0 \in \partial \alpha(0) \cap \partial \beta(0)$.
\end{lemma}
\begin{proof}
Assume that $0 \in \partial f(x,y)$. Since $f$ is the composition of a convex function and a continuously differentiable function, by virtue of \cite[Corollary p. 32]{clarke1990}, \cite[2.3.6 Proposition (b)]{clarke1990}, and \cite[2.3.10 Chain Rule II]{clarke1990}, it is subdifferentially regular \cite[2.3.4 Definition]{clarke1990}. It follows from \cite[2.3.15 Proposition]{clarke1990} that $\partial f(x,y) \subset \partial_{x_1}f(x,y) \times \hdots \times \partial_{x_m}f(x,y) \times \partial_{y_1}f(x,y) \times \hdots \times \partial_{y_n}f(x,y)$. Hence $0 \in \partial_{x_i} f(x,y)$ and $0 \in \partial_{y_j} f(x,y)$ for all indices $i$ and $j$. Based on the expressions of the partial Clarke subdifferentials in \eqref{subdif_x}-\eqref{subdif_y}, we get that $0 \in \partial \alpha (x_i/u_i) \cap \partial \beta (y_j/v_j)$ for any indices $i$ and $j$ such that $u_i\neq 0$ and $v_j \neq 0$. In other words, whenever they are well-defined, the ratios $x_i/u_i$ and $y_j/v_j$ are roots of $\partial \alpha$ and $\partial \beta$, respectively.

We next reason by contradiction and assume that $f(x,y)>0$ and $0 \notin \partial \alpha(0) \cap \partial \beta(0)$. If $u=0$ and $v=0$, then from \eqref{eq:step_functions} we get that $\partial \alpha(t) = \|y\|_1 \mathrm{sign}(t)$ and $\partial \beta(t) = \|x\|_1 \mathrm{sign}(t)$, so that $0 \in \partial \alpha (0) \cap \partial \beta(0)$. As a result, either $u \neq 0$ or $v \neq 0$. Assume that $u=0$. Based on what was just said, $v \neq 0$. Also, $x\neq 0$ and $y\neq 0$, otherwise $f(x,y)=0$. Since $u=0$, from the expression of partial Clarke subdifferential in \eqref{subdif_x} we find that $0 \in \partial f_{x_i}(x,y) = \mathrm{sign}(x_i)\|y\|_1$ for all index $i$. Thus $x=0$, which is a contradiction. We deduce that $u \neq 0$, and by the same reasoning, $v\neq 0$. Assume that $x=0$. Then, from \eqref{eq:step_functions}, we get that $\partial \beta (t) = 0$ for all $t\in \mathbb{R}$, and in particular, $0 \in \partial \beta (0)$. In addition, $x_i/u_i$ is a root of $\partial \alpha$ whenever $u_i\neq 0$, and since $u\neq 0$, $0$ is a root of $\partial \alpha$. As result, $0 \in \partial \alpha(0) \cap \partial \beta(0)$, which is a contradiction. We deduce that $x \neq 0$, and by the same reasoning, $y\neq 0$. In addition, for all indices $i$ and $j$ such that $u_i=0$ and $v_j=0$, we have $0 \in \partial f_{x_i}(x,y) = \mathrm{sign}(x_i)\|y\|_1$ and $0 \in \partial f_{y_j}(x,y) = \mathrm{sign}(y_j)\|x\|_1$, whence $x_i = 0$ and $y_j = 0$. To sum up, so far we have shown that $u\neq 0$, $v\neq 0$, $x\neq 0$, $y\neq 0$, $x_i = 0$ if $u_i = 0$, and $y_j = 0$ if $v_j = 0$.

We next analyze the roots and jumps of $\partial \alpha$ and $\partial \beta$. Neither step function has a jump at the origin, otherwise there exists $i$ and $j$ such that $y_j \neq 0$ yet $v_j = 0$, and $x_i \neq 0$ yet $u_i = 0$. As for the values of the step functions at the origin, they cannot both be zero otherwise $0 \in \partial \alpha (0) \cap \partial \beta(0)$. Without loss of generality, we may thus assume from now on that $\partial \alpha(0) \neq 0$. If the non-decreasing function $\partial \alpha$ has a negative and a positive root, then $\partial \alpha (0) = 0$. As a result, without loss of generality, we may assume that $x_i/u_i>0$ for all index $i$ such that $u_i \neq 0$ (and in particular, $x_i \neq 0$ if $u_i \neq 0$). If $\partial \alpha$ has no positive jump point that is less than or equal to each root $x_i/u_i$, then $\partial \alpha (0) = 0$. Hence, let $v_{j_0}/y_{j_0}$ be such a jump point, where $y_{j_0} \neq 0$. We thus have $0<v_{j_0}/y_{j_0}\leqslant x_i/u_i$ for all index $i$ such that $u_i\neq 0$, as illustrated in Figure \ref{proof_a}.

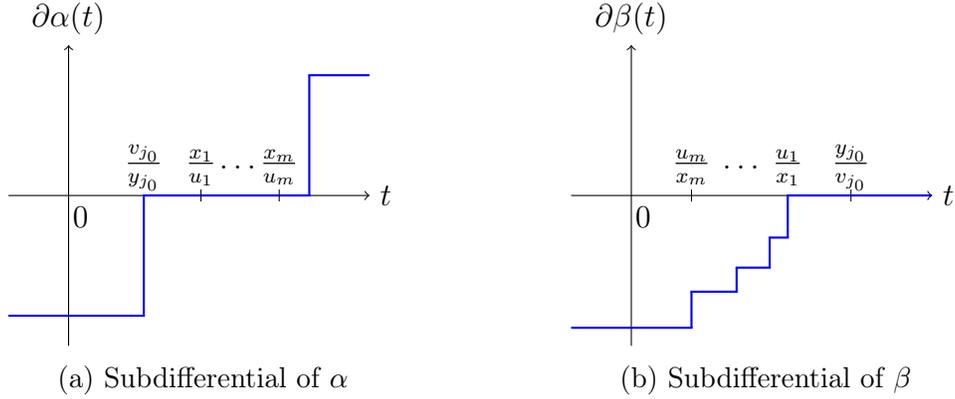
\begin{figure}[ht!]
\centering
\begin{subfigure}{.49\textwidth}
  \centering
  \begin{tikzpicture}[scale=0.8]
\draw[->] (-1,2)--(5,2) node[right]{$t$};
\draw[->] (0,-.5)--(0,4.5) node[above]{$\partial \alpha(t)$};
\draw[blue,thick] (-1,0)--(1.2675,0) ;
\draw[blue,thick] (1.25,0)--(1.25,2) ;
\draw[blue,thick] (1.2325,2)--(4.0175,2) ;
\draw[blue,thick] (4,2)--(4,4);
\draw[blue,thick] (3.9825,4)--(5,4) ;
\draw (1.25,2.47) node {$\frac{v_{j_0}}{y_{j_0}}$};
\draw (2.2,1.9)--(2.2,2.1);
\draw (2.2,2.47) node {$\frac{x_1}{u_1}$};
\draw (2.85,2.35);
\draw (2.81,2.47) node {$\hdots$};
\draw (3.5,1.9)--(3.5,2.1);
\draw (3.5,2.47) node {$\frac{x_m}{u_m}$};
\fill (.2,2)  node[below]{0};
\end{tikzpicture}
  \caption{Subdifferential of $\alpha$}
  \label{proof_a}
\end{subfigure}
\begin{subfigure}{.49\textwidth}
  \centering
  \begin{tikzpicture}[scale=0.8]
\draw[->] (-1,2)--(5,2) node[right]{$t$};
\draw[->] (0,-.5)--(0,4.5) node[above]{$\partial \beta(t)$};
\draw[blue,thick] (-1,-.2)--(1.0175,-.2) ;
\draw[blue,thick] (1,-.2)--(1,0.4) ;
\draw[blue,thick] (0.9825,0.4)--(1.7675,0.4) ;
\draw[blue,thick] (1.75,0.4)--(1.75,.8) ;
\draw[blue,thick] (1.7325,.8)--(2.3175,.8) ;
\draw[blue,thick] (2.3,.8)--(2.3,1.3) ;
\draw[blue,thick] (2.2825,1.3)--(2.6175,1.3) ;
\draw[blue,thick] (2.6,1.3)--(2.6,2) ;
\draw[blue,thick] (2.5825,2)--(4.9825,2) ;
\draw (3.65,1.9)--(3.65,2.1);
\draw (3.65,2.47)  node {$\frac{y_{j_0}}{v_{j_0}}$};
\draw (1,1.9)--(1,2.1);
\draw (1,2.47) node {$\frac{u_m}{x_m}$};
\draw (1.82,2.47) node {$\hdots$};
\draw (2.6,2.47) node {$\frac{u_1}{x_1}$};
\fill (.2,2)  node[below]{0};
\end{tikzpicture}
  \caption{Subdifferential of $\beta$}
  \label{proof_b}
\end{subfigure}
\caption{Visualization of the proof of Lemma \ref{lemma:step}.}
\end{figure}

We next discuss the repercussions of the jump point $v_{j_0}/y_{j_0}$ of $\partial \alpha$ on the roots and jumps of $\partial \beta$. Inverting yields that $0<u_i/x_i\leqslant y_{j_0}/v_{j_0}$ for all index $i$ such that $u_i\neq 0$. As was shown in the second paragraph, $x_i=0$ if $u_i=0$, hence the previous sentence is true for all index $i$ such that $x_i \neq 0$. Hence all the jump points of $\partial \beta$ are less than or equal to one of its roots $y_{j_0}/v_{j_0}$. This is illustrated in Figure \ref{proof_b}. Assume that $\partial \beta (y_{j_0}/v_{j_0}) \subset (-\infty,0]$, which is true if $y_{j_0}/v_{j_0}$ is not a jump, as it is represented in Figure \ref{proof_b}. Then $\partial \beta(t) = 0$ for all $t \geqslant y_{j_0}/v_{j_0}$. As remarked in parenthesis in the previous paragraph, $u_i=0$ if $x_i=0$, hence for all $t$ large enough we have $\mathrm{sign}(x_it-u_i) = \mathrm{sign}(x_it)$. Thus $\partial \beta (t) = \sum_{i=1}^m \mathrm{sign}(x_it)x_i = \sum_{i=1}^m |x_i|$ for all $t$ large enough. It follows that $x=0$, which is contradiction. As a result, there exists $\epsilon>0$ such that $[0,\epsilon] \subset \partial \beta (y_{j_0}/v_{j_0})$, and in particular, $y_{j_0}/v_{j_0}$ is a jump point of $\partial \beta$. Since $\partial \beta$ is a non-decreasing step function, it has no roots greater than $y_{j_0}/v_{j_0}$.

We next consider the case where $[-\epsilon,\epsilon] \subset \partial \beta (y_{j_0}/v_{j_0})$, after possibly reducing $\epsilon>0$. Since $\partial \beta$ is a non-decreasing step function, it has no roots less than $y_{j_0}/v_{j_0}$. Since $y_{j}/v_{j}$ is a root of $\partial \beta$ whenever $v_{j}\neq 0$, it follows that $y_{j}/v_{j} = y_{j_0}/v_{j_0}>0$ for all $v_{j}\neq 0$ (and in particular, we have $y_j \neq 0$ if $v_j \neq 0$). Inverting yields that $v_j/y_j = v_{j_0}/y_{j_0}$ for all $v_j\neq 0$. As was shown in the second paragraph, if $y_j \neq 0$, then $v_j \neq 0$. It follows that $v_j/y_j = v_{j_0}/y_{j_0}$ for all $y_j\neq 0$. Hence $v_{j_0}/y_{j_0}$ is the unique jump point of $\partial \alpha$. Recall that $0<v_{j_0}/y_{j_0}\leqslant x_i/u_i$ and $x_i/u_i$ is a root of $\partial \alpha$ for all index $i$ such that $u_i\neq 0$. Assume that there exists an index $i_0$ such that $x_{i_0}/u_{i_0}$ such that $v_{j_0}/y_{j_0} < x_{i_0}/u_{i_0}$. Then $\partial \alpha(t) = 0$ for all $t \geqslant x_{i_0}/u_{i_0}$. As remarked in parenthesis in this paragraph, it holds that $v_j=0$ if $y_j=0$, and thus for all $t$ large enough we have $\mathrm{sign}(y_jt-v_j) = \mathrm{sign}(y_jt)$. Hence $\partial \alpha (t) = \sum_{j=1}^n \mathrm{sign}(y_jt)y_j = \sum_{j=1}^n |y_j|$ for all $t$ large enough. It follows that $y=0$, which is contradiction. As a result, $v_{j}/y_{j} = v_{j_0}/y_{j_0} = x_i/u_i$ for all indices $i$ and $j$ such that $u_i\neq 0$ and $y_j \neq 0$, and hence $x_iy_j-u_iv_j = 0$. Recall that $x_i=0$ if $u_i=0$, and $v_j = 0$ if $y_j = 0$. Hence $x_iy_j-u_iv_j = 0$ for all indices $i$ and $j$, which implies that $f(x,y)=0$. This is a contradiction.

It remains to consider the case where $\partial \beta (y_{j_0}/v_{j_0}) = [0,\epsilon]$, possibly after increasing $\epsilon>0$. Recall that all the jump points of $\partial \beta$, i.e. $u_i/x_i$ for all index $i$ such that $x_i\neq 0$, are less than or equal to $y_{j_0}/v_{j_0}$. Assume that $u_i/x_i = y_{j_0}/v_{j_0}$ for all index $i$ such that $x_i\neq 0$. Hence $y_{j_0}/v_{j_0}$ is unique jump point of $\partial \beta$. Since $\partial \beta (y_{j_0}/v_{j_0}) = [0,\epsilon]$, we have $\partial \beta (t) = 0$ for all $t\leqslant y_{j_0}/v_{j_0}$. Since $u_i=0$ if $x_i=0$, for all $t$ small enough (i.e., taking large negative values), we have $\mathrm{sign}(x_it-u_i) = -\mathrm{sign}(x_i)$. Hence $\partial \beta (t) = \sum_{i=1}^m -\mathrm{sign}(x_i)x_i = -\sum_{i=1}^m |x_i|$ for all $t$ small enough. It follows that $x=0$, which is contradiction. As a result, there exists an index $i_0$ such that $0 < u_{i_0}/x_{i_0} < y_{j_0}/v_{j_0}$. Recall that $y_j/v_j$ is a root of $\partial \beta$ for all index $j$ such that $v_j\neq 0$. Since $\partial \beta (y_{j_0}/v_{j_0}) = [0,\epsilon]$, we have $0 < u_{i_0}/x_{i_0} \leqslant y_j/v_j \leqslant y_{j_0}/v_{j_0}$ for all index $j$ such that $v_j\neq 0$ (and in particular, we have $y_j \neq 0$ if $v_j \neq 0$). Inverting yields $0 < v_{j_0}/y_{j_0} \leqslant v_j/y_j \leqslant x_{i_0}/u_{i_0}$ for all index $j$ such that $v_j\neq 0$. As was shown in the second paragraph, if $y_j \neq 0$ then $v_j \neq 0$, hence the previous sentence is true for all index $j$ such that $y_j \neq 0$. Since $\partial \beta (y_{j_0}/v_{j_0}) = [0,\epsilon]$, $y_{j_0}/v_{j_0}$ is a jump point of $\partial \beta$, so there exists $i_1$ such that $y_{j_0}/v_{j_0} = u_{i_1}/x_{i_1}$. We thus have $0 <  x_{i_1}/u_{i_1} = v_{j_0}/y_{j_0} \leqslant v_j/y_j \leqslant x_{i_0}/u_{i_0}$ for all index $j$ such that $y_j\neq 0$. If $x_{i_1}/u_{i_1} = v_{j_0}/y_{j_0} < v_j/y_j < x_{i_0}/u_{i_0}$, then $v_j/y_j$ is a jump point of $\partial \alpha$ located strictly between two of its roots, which is impossible. Hence for all index $j$ such that $y_j\neq 0$, either $v_j/y_j = \mu := x_{i_1}/u_{i_1}$ or $v_j/y_j = \nu := x_{i_0}/u_{i_0}$. If $v_j/y_j \neq \nu$ for all index $j$ such that $y_j \neq 0$, then one of the roots of $\partial \alpha$ is greater than all its jump points. Hence $\partial \alpha(t) = 0$ for all $t$ large enough. Together with the fact that $u_i=0$ if $x_i=0$, this yields a contradiction. Hence there exists $j_1$ such that $v_{j_1}/y_{j_1} = x_{i_0}/u_{i_0}$. We next show that the dichotomy $v_j/y_j = \mu$ or $v_j/y_j = \nu$ also holds for $x_i/u_i$ whenever $u_i\neq 0$.

Based on the previous paragraph, we have $0 <  x_{i_1}/u_{i_1} = v_{j_0}/y_{j_0} < v_{j_1}/y_{j_1} = x_{i_0}/u_{i_0}$. Inverting yields $u_{i_0}/x_{i_0} = y_{j_1}/v_{j_1} < y_{j_0}/v_{j_0} = u_{i_1}/x_{i_1}$. Recall that $y_j/v_j$ are roots of $\partial \beta$ for all index $j$ such that $v_j \neq 0$, and the jump points of $\partial \beta$ are $u_i/x_i$ for all index $i$ such that $x_i \neq 0$. The non-decreasing function $\partial \beta$ cannot have a jump point strictly between two of its roots, hence for all index $i$ such that $x_i \neq 0$, it holds that $0<u_i/x_i \leqslant u_{i_0}/x_{i_0} = y_{j_1}/v_{j_1} < y_{j_0}/v_{j_0} = u_{i_1}/x_{i_1}$ or $0< u_{i_0}/x_{i_0} = y_{j_1}/v_{j_1} < y_{j_0}/v_{j_0} = u_{i_1}/x_{i_1} \leqslant u_i/x_i$. Inverting yields $ x_{i_1}/u_{i_1} = v_{j_0}/y_{j_0} < v_{j_1}/y_{j_1} = x_{i_0}/u_{i_0} \leqslant x_i/u_i$ or $x_i/u_i \leqslant x_{i_1}/u_{i_1} = v_{j_0}/y_{j_0} < v_{j_1}/y_{j_1} = x_{i_0}/u_{i_0}$. If the inequality is strict in the first case, then $x_i/u_i$ is a root of $\partial \alpha$ that is greater than all of its jump points, of which there are two according to the previous paragraph, namely $v_{j_1}/y_{j_1}$ and $v_{j_0}/y_{j_0}$. If the inequality is strict in the second case, then $x_i/u_i$ is a root of $\partial \alpha$ that is less than all of its jump points. In either case, $\partial \alpha(t)= 0$ for all $t$ either small or large enough. Together with the fact that $v_j = 0$ if $y_j=0$ (as remarked in parenthesis in the previous paragraph), this yields a contradiction. Hence for all index $i$ such that $u_i\neq 0$, either $x_i/u_i = \mu$ or $x_i/u_i = \nu$, and each case clearly happens for some index (the first with $i_1$, the second with $i_0$). In light of the dichotomy exposed above, consider $(h,k) \in \mathbb{R}^m \times \mathbb{R}^n$ defined by
\begin{equation}
\label{hk}
    h_i := \left\{
    \begin{array}{cl}
        0 & \text{if}~ x_i/u_i = \mu, \\
        -u_i \nu & \text{if}~ x_i/u_i = \nu, \\
        0 & \text{if}~ u_i = 0,
    \end{array}
    \right.
    ~~~~~~\text{and}~~~~~~
    k_j := \left\{
    \begin{array}{cl}
         0 & \text{if}~ y_j/v_j = 1/\mu, \\
        v_j / \nu & \text{if}~ y_j/v_j = 1/\nu, \\
        0 & \text{if}~ v_j = 0.
    \end{array}
    \right.
\end{equation}
Consider also the function $\gamma$ defined from $\mathbb{R}^{m\times n}$ to $\mathbb{R}$ by $\gamma(Q) := h^T Q y + x^T Q k$. So far in the proof, we have only used the necessary condition for being critical provided by \eqref{necessary_critical}. We next invoke the critical condition, as discussed in the introduction. According to the critical condition \eqref{eq:crit}, there exists $\Lambda \in \mathrm{sign}(xy^T-uv^T)$ such that $\Lambda y = 0$ and $\Lambda^Tx = 0$. Thus $\gamma(\Lambda) = h^T \Lambda y + x^T \Lambda k = h^T (\Lambda y) + (\Lambda^T x)^T k = 0$. Yet, observe that the image $\gamma\left(\mathrm{sign}(xy^T-uv^T)\right) = \hdots$
\begin{subequations}
        \begin{align}
             = & \sum_{i=1}^m \sum_{j=1}^n \mathrm{sign}(x_iy_j-u_iv_j)(h_iy_j + x_ik_j) \label{gamma} \\
             = & \sum\limits_{\frac{x_i}{u_i} = \mu} \sum\limits_{\frac{y_j}{v_j} = \frac{1}{\mu}} \mathrm{sign}(x_iy_j-u_iv_j)(0 \times y_j + x_i\times 0) ~ + \label{hkb}\\
            & \sum\limits_{\frac{x_i}{u_i} = \mu} \sum\limits_{\frac{y_j}{v_j} = \frac{1}{\nu}} \mathrm{sign}(x_iy_j-u_iv_j)(0 \times y_j + x_i \times v_j /\nu) ~ + \label{hkc}\\
            & \sum\limits_{\frac{x_i}{u_i} = \mu} \sum\limits_{v_j = 0} \mathrm{sign}(x_iy_j-u_iv_j)(0 \times y_j + x_i \times 0) ~ + \label{hkd}\\
        & \sum\limits_{\frac{x_i}{u_i} = \nu} \sum\limits_{\frac{y_j}{v_j} = \frac{1}{\mu}} \mathrm{sign}(x_iy_j-u_iv_j)(-u_i\nu \times y_j + x_i \times 0) ~ + \label{hke}\\
        & \sum\limits_{\frac{x_i}{u_i} = \nu} \sum\limits_{\frac{y_j}{v_j} = \frac{1}{\nu}} \mathrm{sign}(x_iy_j-u_iv_j)(-u_i\nu \times y_j + x_i \times v_j/\nu) ~+\label{hkf}\\
        & \sum\limits_{\frac{x_i}{u_i} = \nu} \sum\limits_{v_j = 0} \mathrm{sign}(x_iy_j-u_iv_j)(-u_i\nu \times y_j + x_i \times 0) ~+\label{hkg}\\
        & \sum\limits_{u_i = 0} \sum\limits_{\frac{y_j}{v_j} = \frac{1}{\mu}} \mathrm{sign}(x_iy_j-u_iv_j)(0 \times y_j + x_i \times 0) ~ + \label{hkh}\\
        & \sum\limits_{u_i = 0} \sum\limits_{\frac{y_j}{v_j} = \frac{1}{\nu}} \mathrm{sign}(x_iy_j-u_iv_j)(0 \times y_j + x_i \times v_j/\nu)~ + \label{hki}\\
        & \sum\limits_{u_i = 0} \sum\limits_{v_j = 0} \mathrm{sign}(x_iy_j-u_iv_j)(0 \times y_j + x_i \times 0) \label{hkj}\\
        = & \sum\limits_{\frac{x_i}{u_i} = \mu} \sum\limits_{\frac{y_j}{v_j} = \frac{1}{\nu}} \mathrm{sign}((\mu/\nu-1)u_iv_j) (u_i \mu v_j/\nu) ~ + \label{xy1}\\
        & \sum\limits_{\frac{x_i}{u_i} = \nu} \sum\limits_{\frac{y_j}{v_j} = \frac{1}{\mu}} \mathrm{sign}((\nu/\mu-1)u_iv_j)(-u_i\nu v_j/\mu)\label{xy2}\\
        = & - \mu/\nu \sum\limits_{\frac{x_i}{u_i} = \mu} \sum\limits_{\frac{y_j}{v_j} = \frac{1}{\nu}} |u_i v_j| - \nu/\mu \sum\limits_{\frac{x_i}{u_i} = \nu} \sum\limits_{\frac{y_j}{v_j} = \frac{1}{\mu}} |u_iv_j| ~<~ 0. \label{final}
        \end{align}
\end{subequations}
Above, \eqref{gamma} follows from the definition of the function $\gamma$. We substitute $h_i$ and $k_j$ using their definition in \eqref{hk}, which yields \eqref{hkb}-\eqref{hkj}. We next substitute $x_i$ and $y_j$ using their expressions below the summation signs and obtain \eqref{xy1}-\eqref{xy2}. Note that all but two terms cancel out: \eqref{hkb} cancels out for obvious reasons; \eqref{hkc} yields \eqref{xy1}; \eqref{hkd} cancels out for obvious reasons; \eqref{hke} yields \eqref{xy2}; \eqref{hkf} cancels out because $-u_i\nu \times y_j + x_i \times v_j/\nu = -u_i\nu \times v_j/\nu + u_i \nu \times v_j/\nu = 0$; \eqref{hkg} cancels out because $y_j=0$ if $v_j=0$, as shown in the second paragraph; \eqref{hkh} cancels out for obvious reasons; \eqref{hki} cancels out because $x_i = 0$ if $u_i=0$, as was shown in the second paragraph; \eqref{hkj} cancels out for obvious reasons. To get from \eqref{xy1}-\eqref{xy2} to \eqref{final}, we use the fact that $\mu/\nu-1<0$ and $\nu/\mu-1>0$ since $0 < \mu < \nu$. We also use the fact that $\mathrm{sign}(u_iv_j)u_iv_j = |u_iv_j|$. The result in \eqref{final} is negative because the summation takes place over non-empty sets: $x_{i_0}/u_{i_0} = \mu$, $y_{j_0}/v_{j_0} = 1/\nu$, $x_{i_1}/u_{i_1} = \nu$, and $y_{j_1}/v_{j_1} = 1/\mu$. (While it may seem strange at first, the image of the set $\mathrm{sign}(xy^T-uv^T)$ via the function $\gamma$ is actually a singleton.) In particular, $\gamma(\Lambda) < 0$ whereas we had shown above that $\gamma(\Lambda) = 0$. This is a contradiction and terminates the proof. 
\end{proof}

Thanks to Lemma \ref{lemma:step}, we may now remove the existence quantifier from the critical condition \eqref{eq:crit}. We obtain a finite number of unions and intersections of polynomial equations and inequalities, where the number is independent of the dimensions $m$ and $n$.
\begin{lemma}
\label{lemma:eliminate}
$(x,y) \in \mathbb{R}^m \times  \mathbb{R}^n$ is a critical point of $f$ if and only if
\begin{subequations}
    \begin{gather}
        x_iy_j = u_iv_j ~~ \text{for}~~i=1,\hdots,m,~j=1,\hdots,n, ~~~\text{or} \label{eliminate_a} \\[2mm]
         \left|\sum\limits_{u_i\neq 0} \mathrm{sign}(u_i) x_i\right| \leqslant \sum\limits_{u_i= 0} |x_i| ~~~ \text{and} ~~~ y = 0, ~~~\text{or} \label{eliminate_b} \\
         x = 0 ~~~ \text{and} ~~~ \left|\sum\limits_{v_j\neq 0} \mathrm{sign}(v_j) y_j\right| \leqslant \sum\limits_{v_j = 0} |y_j|, ~~~\text{or} \label{eliminate_c} \\[3mm]
         \sum\limits_{u_i\neq 0} \mathrm{sign}(u_i) x_i = \sum\limits_{v_j\neq 0} \mathrm{sign}(v_j) y_j = 0, ~~~~
        \frac{x_iy_j}{u_iv_j} \leqslant 1~~ \text{if}~~u_iv_j\neq 0, \label{eliminate_d} \\[3mm]  x_i = 0~~\text{if}~~u_i=0,~~~\text{and}~~~ y_j= 0~~\text{if}~~v_j=0. \label{eliminate_e}
    \end{gather}
\end{subequations}
\end{lemma}
\begin{proof} ($\Longrightarrow$)
Assume that $0 \in \partial f(x,y)$. If $f(x,y) = 0$, then \eqref{eliminate_a} holds. Otherwise, we next prove that \eqref{eliminate_b}-\eqref{eliminate_e} hold. From the expression of $\partial \alpha (t)$ and $\partial \beta(t)$ in \eqref{eq:step_functions}, it follows that 
\begin{subequations}
    \begin{align}
        \partial \alpha(0) = -\sum_{j=1}^n  \mathrm{sign}(v_j)y_j = -\sum_{v_j\neq 0} \mathrm{sign}(v_j)y_j - \sum_{v_j = 0} \mathrm{sign}(0)y_j, \label{origin_a} \\
        \partial \beta(0) = - \sum\limits_{i=1}^m \mathrm{sign}(u_i)x_i = - \sum_{u_i \neq 0} \mathrm{sign}(u_i)x_i - \sum_{u_i = 0} \mathrm{sign}(0)x_i. \label{origin_b}
\end{align}
\end{subequations}
Since $0 \in \partial f(x,y)$ and $f(x,y) \neq 0$, from Lemma \ref{lemma:step} we get that $0 \in \partial \alpha(0) \cap \partial \beta(0)$. Together with \eqref{origin_a}-\eqref{origin_b}, this yields that
\begin{equation}
\label{inequalities}
    \left|\sum\limits_{u_i\neq 0} \mathrm{sign}(u_i) x_i\right| \leqslant \sum\limits_{u_i= 0} |x_i| ~~~\text{and}~~~ \left|\sum\limits_{v_j\neq 0} \mathrm{sign}(v_j) y_j\right| \leqslant \sum\limits_{v_j = 0} |y_j|.
\end{equation}
If $y=0$, then from the inequality on the left hand side of \eqref{inequalities}, we obtain \eqref{eliminate_b}. Likewise, if $x=0$, then from the inequality on the right hand side of \eqref{inequalities}, we obtain \eqref{eliminate_c}. If neither $x=0$ nor $y=0$, then for all indices $i$ and $j$ such that $u_i=0$ and $v_j=0$, we have $0 \in \partial_{x_i} f(x,y) = \mathrm{sign}(x_i)\|y\|_1$ and $0 \in \partial_{y_j} f(x,y) = \mathrm{sign}(y_j)\|x\|_1$, whence $x_i=0$ and $y_j=0$. Hence \eqref{eliminate_e} is true. In light of this, the inequalities in \eqref{inequalities} become equalities and equal to zero, yielding the equalities in \eqref{eliminate_d}. It remains to show the ratio inequalities in \eqref{eliminate_d}, for which we don't need to assume that neither $x$ nor $y$ are equal to zero. Indeed, the ratio inequalities trivially hold in this case. The ratio inequalities are the object of the next paragraph.

Since $0 \in \partial_{x_i} f(x,y)$ and $0 \in \partial_{y_j} f(x,y)$ for all indices $i$ and $j$, from the expressions of the partial Clarke sudifferentials in \eqref{subdif_x}-\eqref{subdif_y}, it follows that $0 \in \partial \alpha (x_i/u_i) \cap \beta (y_j/v_j)$ whenever $u_i\neq 0$ and $v_j \neq 0$. Recall from Lemma \ref{lemma:step} that we also have that $0 \in \partial \alpha(0) \cap \partial \beta(0)$. Observe that the non-decreasing function $\partial \alpha$ cannot contain a jump point between the root $0$ and any root $x_i/u_i$. Hence, for all index $j$ such that $y_j \neq 0$, if the jump point $v_j/y_j$ is positive, then it is greater than or equal to all the roots $x_i/u_i$, that is to say, $v_j/y_j \geqslant x_i/u_i$. If the jump point $v_j/y_j$ is negative, then it is less than or equal to all the roots $x_i/u_i$, that is to say, $v_j/y_j \leqslant x_i/u_i$. Multiplying both inequalities by $y_j/v_j$ yields $x_iy_j/(u_iv_j) \leqslant 1$ whenever $u_iv_j \neq 0$.

($\Longleftarrow$) If \eqref{eliminate_a} holds, then $(x,y)$ is global minimum of $f$ and hence a critical point according to the generalized Fermat rule \cite[2.3.2 Proposition]{clarke1990}. If \eqref{eliminate_b} holds, then from the expression of $\partial \beta(0)$ in \eqref{origin_b}, it follows that $0 \in \partial \beta(0)$. From the expression of $\partial_{y_j} f(x,y)$ in \eqref{subdif_y} and $y=0$, we then get that $0 \in \partial_{y_j} f(x,y)$ for all index $j$. Since $f(x,\cdot)$ is subdifferentially regular for all $x \in \mathbb{R}^m$, by virtue of \cite[2.3.15 Proposition]{clarke1990} it holds that $\partial_y f(x,y) \subset \partial_{y_1} f(x,y) \times \hdots \times \partial_{y_n} f(x,y)$. Using \cite[2.3.10 Chain Rule II]{clarke1990}, we actually find that $\partial_y f(x,y) = \partial_{y_1} f(x,y) \times \hdots \times \partial_{y_n} f(x,y) = \{ \Lambda^T x ~|~ \Lambda \in \mathrm{sign}(xy^T-uv^T)\}$. Thus there exists $\Lambda \in \mathrm{sign}(xy^T-uv^T)$ such that $\Lambda^T x = 0$. Since $y = 0$, we naturally also have that $\Lambda y = 0$. Hence the critical condition \eqref{eq:crit} holds. The same argument applies when \eqref{eliminate_c} holds. Assume that \eqref{eliminate_d}-\eqref{eliminate_e} hold. In order to exhibit a matrix $\Lambda$ in the critical condition \eqref{eq:crit}, we propose to define the following function: 
\begin{equation}
\mathrm{sgn}(t) :=
\left\{
\begin{array}{rl}
-1 & \text{if} ~ t < 0, \\
0 & \text{if} ~ t = 0, \\
1 & \text{if} ~ t > 0.
\end{array}
\right.
\end{equation}
Like the $\mathrm{sign}$ function, the $\mathrm{sgn}$ function applies to matrices term by term. Using this new function, we may rewrite the equalities in \eqref{eliminate_d} as 
\begin{subequations}
\begin{align}
    \sum\limits_{j=1}^n \mathrm{sgn}(v_j) y_j &= 0,~~~ i = 1,\hdots,m, \\
    \sum\limits_{i=1}^m \mathrm{sgn}(u_i) x_i &= 0,~~~ j = 1,\hdots,n.
\end{align}
\end{subequations}
After multiplying the above equations by $\mathrm{sgn}(u_i)$ and $\mathrm{sgn}(v_j)$ respectively, we obtain that 
\begin{subequations}
\begin{align}
    \sum\limits_{j=1}^n \mathrm{sgn}(u_iv_j) y_j &= 0,~~~ i = 1,\hdots,m, \\
    \sum\limits_{i=1}^m \mathrm{sgn}(v_ju_i) x_i &= 0,~~~ j = 1,\hdots,n.
\end{align}
\end{subequations}
In other words, $\Lambda y = 0 $ and $\Lambda^T x = 0$ with $\Lambda := -\mathrm{sgn}(uv^T)$. In order to show that $(x,y)$ is critical, it remains to show that $\Lambda \in \mathrm{sign}(xy^T-uv^T)$. If $u_iv_j=0$, then from \eqref{eliminate_e} we get that $x_iy_j = 0$. In that case, $-\mathrm{sgn}(u_iv_j) = 0$ and $\mathrm{sign}(x_iy_j-u_iv_j) = [-1,1]$. Hence $-\mathrm{sgn}(u_iv_j) \in \mathrm{sign}(x_iy_j-u_iv_j)$. If $u_iv_j \neq 0$, then from the inequalities in \eqref{eliminate_d} we get that $x_iy_j/(u_iv_j)-1 \leqslant 0$. In that case, $\mathrm{sign}(x_iy_j-u_iv_j) = \mathrm{sign}(u_iv_j)\mathrm{sign}(x_iy_j/(u_iv_j)-1) \ni -\mathrm{sign}(u_iv_j)  = -\mathrm{sgn}(u_iv_j)$.
\end{proof}

Thanks to Lemma \ref{lemma:eliminate}, we may now classify the critical points according to whether they are global minima, local minima that are not global minima (i.e. spurious local minima), or not local minima (i.e. saddle points).
\begin{proposition}
\label{prop:classify}
$f$ has the following properties when $u \neq 0$ and $v \neq 0$.
\begin{enumerate}[label={\arabic*)}]
    \item The global minima are all $(x,y) \in \mathbb{R}^m\times\mathbb{R}^n$ such that 
\begin{equation}
      \exists \theta \in \mathbb{R} \setminus \{0\}: ~~~  (x,y) = (u\theta,v/\theta) \label{eq:global}
\end{equation}
    \item The spurious local minima are all $(x,y) \in \mathbb{R}^m\times\mathbb{R}^n$ such that
    \begin{subequations}
\begin{gather}
     ~~~~~~\left|\sum\limits_{u_i\neq 0} \mathrm{sign}(u_i) x_i\right| < \sum\limits_{u_i= 0} |x_i|~~\text{and}~~ y=0,~~~~\text{or} \label{eq:spurious_a} \\
     x=0 ~~\text{and}~~\left|\sum\limits_{v_j\neq 0} \mathrm{sign}(v_j) y_j\right| < \sum\limits_{v_j = 0} |y_j|. \label{eq:spurious_b}
\end{gather}
\end{subequations}
    \item The saddle points are all $(x,y) \in \mathbb{R}^m\times\mathbb{R}^n$ such that
    \begin{subequations}
    \begin{gather}
         \left|\sum\limits_{u_i\neq 0} \mathrm{sign}(u_i) x_i\right| = \sum\limits_{u_i= 0} |x_i| ~~~ \text{and} ~~~ y = 0, ~~~\text{or} \label{saddle_a} \\
         x = 0 ~~~ \text{and} ~~~ \left|\sum\limits_{v_j\neq 0} \mathrm{sign}(v_j) y_j\right| = \sum\limits_{v_j = 0} |y_j|, ~~~\text{or} \label{saddle_b} \\[3mm]
         \sum\limits_{u_i\neq 0} \mathrm{sign}(u_i) x_i = \sum\limits_{v_j\neq 0} \mathrm{sign}(v_j) y_j = 0, ~~~~
        \frac{x_iy_j}{u_iv_j} \leqslant 1~~ \text{if}~~u_iv_j\neq 0, \label{saddle_c} \\[3mm]  x_i = 0~~\text{if}~~u_i=0,~~~\text{and}~~~ y_j= 0~~\text{if}~~v_j=0. \label{saddle_d}
    \end{gather}
\end{subequations}
\end{enumerate}
In addition, for all saddle point $(x,y)$, there exists a global minimum $(x^*,y^*)$ such that $(x^*,y^*)-(x,y)$ is a direction of descent. 
\end{proposition}
\begin{proof}
Lemma \ref{lemma:eliminate} implies that \eqref{eq:global}, \eqref{eq:spurious_a}-\eqref{eq:spurious_b}, and \eqref{saddle_a}-\eqref{saddle_d} form a partition of the set of critical points. It thus suffices to check in each case that the desired property about the local variation of $f$ holds true.
1) Observe that \eqref{eliminate_a} and \eqref{eq:global} are equivalent because $u\neq 0$ and $v\neq 0$. 2) Consider $x\in \mathbb{R}^m$ and $y=0$. For all $(h,k) \in \mathbb{R}^m \times \mathbb{R}^n$ small enough, we have $f(x+h,y+k) = \hdots $
\begin{subequations}
    \begin{align}
      \label{brutal_a} = & \sum\limits_{i=1}^m \sum\limits_{j=1}^n |(x_i+h_i)k_j - u_iv_j| \\
      \label{brutal_b} = & \sum\limits_{u_iv_j=0}  |(x_i+h_i)k_j| + \sum\limits_{u_iv_j\neq 0}  |(x_i+h_i)k_j - u_iv_j| \\
      \label{brutal_c} = & \sum\limits_{u_iv_j=0}  |(x_i+h_i)k_j| + \sum\limits_{u_iv_j\neq 0}  |u_iv_j| - \mathrm{sign}(u_iv_j)(x_i+h_i)k_j \\
    \label{brutal_d}  = & f(x,y) + \sum\limits_{u_iv_j=0}  |(x_i+h_i)k_j| - \sum\limits_{u_iv_j\neq 0}   \mathrm{sign}(u_iv_j)(x_i+h_i)k_j \\
     \label{brutal_e}  = &  f(x,y) + \sum\limits_{u_iv_j=0}  |(x_i+h_i)k_j| - \sum\limits_{v_j\neq 0}k_j\mathrm{sign}(v_j)\sum\limits_{u_i \neq 0} \mathrm{sign}(u_i)(x_i + h_i)\\
     \label{brutal_f}  \geqslant & f(x,y) + \sum\limits_{u_iv_j=0}  |(x_i+h_i)k_j| - \sum\limits_{v_j\neq 0}|k_j| \left|\sum\limits_{u_i \neq 0} \mathrm{sign}(u_i)(x_i + h_i)\right|\\
     & \hspace*{-2.5mm} (\text{equality holds if}~ \mathrm{sign}(k_j) = \mathrm{sign}\left(v_j \sum_{u_i \neq 0} \mathrm{sign}(u_i)(x_i + h_i)\right)~\text{when}~v_j \neq 0) \notag \\
     \label{brutal_g}  = & f(x,y) + \sum\limits_{v_j=0} |k_j|\sum\limits_{i = 1}^n |x_i+h_i| + \hdots \\
     \label{brutal_h}  & \sum\limits_{v_j\neq 0}|k_j| \left( \sum\limits_{u_i = 0}|x_i + h_i| -  \left|\sum\limits_{u_i \neq 0} \mathrm{sign}(u_i)(x_i + h_i)\right| \right)
       \\
     \label{brutal_i}  \geqslant & f(x,y) + \sum\limits_{v_j\neq 0}|k_j| \left( \sum\limits_{u_i = 0}|x_i + h_i| -  \left|\sum\limits_{u_i \neq 0} \mathrm{sign}(u_i)(x_i + h_i)\right| \right)\\
     & \hspace*{2cm} (\text{equality holds if}~k_j = 0~\text{when}~v_j = 0). \notag
    \end{align}
\end{subequations}
Above, \eqref{brutal_a} is a consequence of the definition of $f$ and $y=0$. \eqref{brutal_b} is obtained by splitting the sum according to whether the product $u_i v_j$ is equal to zero. \eqref{brutal_c} is valid since $k_j$ is small and $|a+b|=|a|-\mathrm{sign}(a)b$ if $|a|>|b|$. \eqref{brutal_d} is due to $f(x,0) = \sum_{u_iv_j\neq 0}  |u_iv_j|$ and $y=0$. \eqref{brutal_e} is obtained by writing that $\mathrm{sign}(u_iv_j)(x_i+h_i)k_j = \mathrm{sign}(u_i)(x_i+h_i)\mathrm{sign}(v_j)k_j$, then factorizing the sum. \eqref{brutal_f} is obtained by observing that the last sum in \eqref{brutal_e} is less than or equal to its absolute value. We then apply the triangular inequality and the fact that $|k_j\mathrm{sign}(v_j)| = |k_j|$. Equality in the inequality in \eqref{brutal_f} is then obtained by taking $k_j$ to be of the same sign as the term it multiplies in the summation over $v_j \neq 0$ in \eqref{brutal_e}. \eqref{brutal_g}-\eqref{brutal_h} are obtained by splitting the first sum in \eqref{brutal_f} according to whether the $v_j$ is equal to zero. Finally, inequality \eqref{brutal_i} holds since $\sum_{v_j=0} |k_j|\sum_{i = 1}^n |x_i+h_i| \geqslant 0$. This term is equal to zero if $k_j$ when $v_j = 0$, in which case we get equality in the inequality in \eqref{brutal_i}.

Assume that \eqref{eq:spurious_a} holds. For all $(h,k) \in \mathbb{R}^m \times \mathbb{R}^n$ small enough, \eqref{brutal_i} yields
\begin{subequations}
\begin{align}
    f(x+h,y+k) &\geqslant f(x,y) + \frac{1}{2}\sum\limits_{v_j\neq 0}|k_j|\left(\sum\limits_{u_i = 0 } |x_i| - \left| \sum\limits_{u_i \neq 0} \mathrm{sign}(u_i)x_i \right|\right) \\ 
    & \geqslant f(x,y)>0
    \end{align}
\end{subequations}
where the strict equality is due to $y=0$, $u\neq 0$ and $v\neq 0$. Thus $(x,y)$ is a spurious local minimum. The same argument applies to \eqref{eq:spurious_b}. 

3) Assume that \eqref{saddle_a} holds, namely $|\sum_{u_i\neq 0} \mathrm{sign}(u_i) x_i| = \sum_{u_i= 0} |x_i|$ and $y = 0$. If $| \sum_{u_i \neq 0} \mathrm{sign}(u_i)x_i | = 0$, then \eqref{saddle_c}-\eqref{saddle_d} hold, a case that we will treat later. If $| \sum_{u_i \neq 0} \mathrm{sign}(u_i)x_i | \neq 0$, then take any $\theta \neq 0$ such that $\mathrm{sign}(\theta) = \mathrm{sign}(\sum_{u_i \neq 0}$ $\mathrm{sign}(u_i)x_i)$ and consider the direction $(u \theta, v/\theta) -(x,y)$. It goes from $(x,y)$ towards the global minimum $(u \theta, v/\theta)$. As we next show, taking a small step $t>0$ in this direction renders the inequalities in \eqref{brutal_f} and in \eqref{brutal_i} binding, where $(h,k) := t(u \theta - x, v/\theta - y)$. Regarding \eqref{brutal_f}, observe that, if $v_j \neq 0$, then
\begin{subequations}
    \begin{align}
    	\label{subeq:32_a} \mathrm{sign}(k_j) & =  \mathrm{sign}(tv_j/\theta)\\
       \label{subeq:32_b}  & =  \mathrm{sign}\left(tv_j\sum\limits_{u_i \neq 0} \mathrm{sign}(u_i)x_i\right) \\
       \label{subeq:32_c}  & =  \mathrm{sign}\left((1-t)v_j\sum\limits_{u_i \neq 0} \mathrm{sign}(u_i)x_i\right) \\
        \label{subeq:32_d}  & =  \mathrm{sign} \Biggr( (1-t)v_j\underbrace{\sum\limits_{u_i \neq 0} \mathrm{sign}(u_i)x_i}_{\neq 0} + tv_j\sum\limits_{u_i \neq 0} \mathrm{sign}(u_i)u_i \theta \Biggr) \\
       \label{subeq:32_e}   & =  \mathrm{sign}\left(v_j\sum\limits_{u_i \neq 0} \mathrm{sign}(u_i)[(1-t)x_i + tu_i \theta]\right) \\
        \label{subeq:32_f}  & = \mathrm{sign}\left(v_j\sum\limits_{u_i \neq 0} \mathrm{sign}(u_i)(x_i + h_i)\right).
    \end{align}
\end{subequations}
Above, \eqref{subeq:32_a} is due to $k_j = v_j/\theta - y_j$, $y = 0$, and $t>0$. \eqref{subeq:32_b} follows from $\mathrm{sign}(\theta) = \mathrm{sign}(\sum_{u_i \neq 0} \mathrm{sign}(u_i)x_i)$. Also, since $t>0$ is small, multiplying a number by $t$ or $(1-t)$ doesn't change its sign, hence \eqref{subeq:32_c}. It is legitimate to add a linear term in function of $t$ in \eqref{subeq:32_d} since it is dominated by the first term ($t$ is small). We get \eqref{subeq:32_e} by factorizing the two terms by $v_j$. Finally, \eqref{subeq:32_f} is due to the fact that, by definition, $h = t(u \theta - x)$, and hence $t u_i \theta =h_i + tx_i$. Regarding \eqref{brutal_i}, observe that, if $v_j = 0$, then $k_j=tv_j/\theta = 0$.

In order to show that $(u \theta - x, v/\theta - y)$ is a direction of descent, we pick up the computation in \eqref{brutal_a}-\eqref{brutal_i} where we left off: $f(x+h,y+k) = \hdots$
\begin{subequations}
    \begin{align}
   \label{very_brutal_a} = & f(x,y) + \sum\limits_{v_j\neq 0}|tv_j/\theta| \left( \sum\limits_{u_i = 0}|x_i + t (u_i \theta - x_i)| +  \right. \\
  \label{very_brutal_b}   & \left. -  \left|\sum\limits_{u_i \neq 0} \mathrm{sign}(u_i)[x_i + t(u_i \theta - x_i)]\right| \right)\\
   \label{very_brutal_c}  = & f(x,y) + t/|\theta|\sum\limits_{v_j\neq 0}|v_j| \left( \sum\limits_{u_i = 0}|(1-t)x_i| +  \right. \\
   \label{very_brutal_d}  & \left. -  \left|(1-t)\sum\limits_{u_i \neq 0} \mathrm{sign}(u_i)x_i+t\theta\sum\limits_{u_i \neq 0} \mathrm{sign}(u_i)u_i\right| \right)\\
     \label{very_brutal_e}   = &f(x,y) + t/|\theta|\sum\limits_{j=1}^n |v_j| \left[ (1 - t)\sum\limits_{u_i = 0}|x_i| -  (1-t)\left|\sum\limits_{u_i \neq 0} \mathrm{sign}(u_i)x_i\right| \right. +\\
    \label{very_brutal_f}    & 
      \left. -\mathrm{sign}\left((1-t)\sum\limits_{u_i \neq 0} \mathrm{sign}(u_i)x_i\right)t\theta \sum\limits_{u_i \neq 0} |u_i| \right]\\
      \label{very_brutal_g}  = &f(x,y) + t/|\theta|\sum\limits_{j=1}^n|v_j| \Biggr[ (1 - t) \underbrace{\left( \sum\limits_{u_i = 0}|x_i| - \left|\sum\limits_{u_i \neq 0}\mathrm{sign}(u_i)x_i\right| \right)}_{=~0} + \\[-2mm]
      \label{very_brutal_h} & - t|\theta| \sum\limits_{i=1}^m|u_i| \Biggr]\\[1mm]
    \label{very_brutal_i}    = &f(x,y) - t^2\sum\limits_{j=1}^n|v_j| \sum\limits_{i=1}^m|u_i|\\[1mm]
     \label{very_brutal_j}   = & (1-t^2)f(x,y).
    \end{align}
\end{subequations}
Above, \eqref{very_brutal_a}-\eqref{very_brutal_b} is due to the definition of $(h,k)$. We get \eqref{very_brutal_c}-\eqref{very_brutal_d} by factorizing by $t/|\theta|$, canceling out $u_i$ in the summation over $u_i=0$, and expanding the product in the summation over $u_i\neq 0$. \eqref{very_brutal_e}-\eqref{very_brutal_f} use the fact that $t>0$ is small and $\sum_{u_i \neq 0}\mathrm{sign}(u_i)x_i \neq 0$. They also use the fact that $|a+b| = |a|+\mathrm{sign}(a)b$ if $|a|>|b|$. \eqref{very_brutal_g}-\eqref{very_brutal_h} are due to $\mathrm{sign}(\theta) = \mathrm{sign}(\sum_{u_i \neq 0} \mathrm{sign}(u_i)x_i)$ and $\mathrm{sign}(\theta)\theta = |\theta|$. It also uses the fact the multiplying a number by $1-t>0$ doesn't change its sign. The term that cancels out in \eqref{very_brutal_g} is due to \eqref{saddle_a}. Finally, \eqref{very_brutal_j} uses the definition of $f$ and the fact that $y=0$. It implies that $(h,k)$ is a direction of descent. We conclude that any point $(x,y)$ satisfying \eqref{saddle_a} is a saddle point and that it admits a direction of descent towards the global minimum $(u \theta, v/\theta)$. The same argument applies to \eqref{saddle_b}, namely $x = 0$ and $|\sum_{v_j\neq 0} \mathrm{sign}(v_j) y_j| = \sum_{v_j = 0} |y_j|$. 

We now treat the remaining case \eqref{saddle_c}-\eqref{saddle_d}, namely $\sum_{u_i\neq 0} \mathrm{sign}(u_i) x_i = \sum_{v_j\neq 0} \mathrm{sign}(v_j) y_j = 0$, $x_iy_j/(u_iv_j) \leqslant 1$ if $u_iv_j\neq 0$, $x_i = 0$ if $u_i=0$, and $y_j= 0$ if $v_j=0$. Given $\theta \neq 0$, consider the direction $(u \theta, v/\theta) -(x,y)$, which goes from $(x,y)$ towards the global minimum $(u \theta, v/\theta)$. As we explain below, one can choose $\theta$ such that, when taking a small step $t>0$ in this direction, the ratio inequalities in \eqref{saddle_c} remain valid. In other words, we have
\begin{equation}
\label{eq:ration_preserved_1}
    \frac{(x_i+h_i)(y_j+k_j)}{u_iv_j} \leqslant 1~,~~~\text{if}~~~u_iv_j \neq 0.
\end{equation}
where $(h,k) := t(u \theta -x, v/\theta - y)$ and $t>0$ is small enough. In order to prove this, observe that \eqref{eq:ration_preserved_1} is equivalent to 
\begin{equation}
\label{eq:ration_preserved_2}
    \left((1-t)\frac{x_i}{u_i}+t\theta\right)\left((1-t)\frac{y_j}{v_j}+t/\theta\right) \leqslant 1~,~~~\text{if}~~~u_iv_j \neq 0,
\end{equation}
where we simply use the definition of $(h,k)$. For all indices $i$ and $j$ where $x_iy_j/(u_iv_j) < 1$, the inequality in \eqref{eq:ration_preserved_2} holds by continuity for $t$ small enough, regardless of $\theta$. Hence, if there are no binding inequalities $x_iy_j/(u_iv_j) = 1$, then one may choose any $\theta \neq 0$. Otherwise, the set of binding inequalities $(x_i/u_i)(y_j/v_j) = 1$ can be decomposed into two groups: those with positive ratios (i.e. $x_i/u_i>0$ and $y_j/v_j>0$) and those with negative ratios (i.e. $x_i/u_i$ and $y_j/v_j<0$). Without loss of generality, we may assume that there exists at least one positive binding ratio. As it turns out, all the positive ratios $x_i/u_i$ involved in binding inequalities are equal to one another. Indeed, if $(x_i/u_i)(y_j/v_j) = 1$ and $(x_k/u_k)(y_l/v_l) = 1$, then the inequalities $(x_i/u_i)(y_l/v_l) \leqslant 1$ and $(x_k/u_k)(y_j/v_j) \leqslant 1$ yield $x_i/u_i \leqslant x_k/u_k$ and $x_k/u_k \leqslant x_i/u_i$, that is to say, $x_k/u_k = x_i/u_i$. Let $\theta>0$ denote this common ratio.
Now consider a binding inequality in \eqref{eq:ration_preserved_2}: if $x_i/u_i>0$, then $x_i/u_i = \theta$ and $y_j/v_j = 1/\theta$, so that the inequality in \eqref{eq:ration_preserved_2} readily holds. If $x_i/u_i<0$, then $x_i/u_i<(1-t)x_i/u_i+t\theta<0$ and $y_j/v_j<(1-t)y_j/v_j+t\theta<0$ for $t>0$ small enough, hence $[(1-t)x_i/u_i+t\theta][(1-t)y_j/v_j+t\theta] \leqslant x_iy_j/(u_iv_j) \leqslant 1$. Again, the inequality in \eqref{eq:ration_preserved_2} holds. As a result, we have found $\theta$ for which \eqref{eq:ration_preserved_2} is true with $t>0$ small enough. We next make use of this to compute $f(x+h,y+k)=\hdots$
\begin{subequations}
\label{eq:objective}
	\begin{align}
	\label{last_brutal_a} = & \sum\limits_{i = 1}^m\sum\limits_{j = 1}^n |(x_i+h_i) (y_j+k_j) - u_i v_j|\\
	\label{last_brutal_b} = & \sum\limits_{u_i \neq 0}\sum\limits_{v_j \neq 0} |(x_i+h_i) (y_j+k_j) - u_i v_j|\\
	\label{last_brutal_c} = & \sum\limits_{u_i \neq 0}\sum\limits_{v_j \neq 0} |u_i v_j|\left|\frac{(x_i+h_i) (y_j+k_j)}{u_iv_j} - 1\right|\\
	\label{last_brutal_d} = & \sum\limits_{u_i \neq 0}\sum\limits_{v_j \neq 0} |u_i v_j|\left(1 - \frac{(x_i+h_i) (y_j+k_j)}{u_iv_j}\right)\\
	\label{last_brutal_e} = & \sum\limits_{u_i \neq 0}\sum\limits_{v_j \neq 0} |u_i v_j| - \sum\limits_{u_i \neq 0}\sum\limits_{v_j \neq 0} \mathrm{sign}(u_iv_j)(x_i+h_i) (y_j+k_j)\\
	\label{last_brutal_f} = & \sum\limits_{u_iv_j \neq 0}\sum\limits_{v_j \neq 0} |u_i v_j| - \sum\limits_{u_i \neq 0} \mathrm{sign}(u_i)(x_i+h_i)\sum\limits_{v_j \neq 0}  \mathrm{sign}(v_j)(y_j+k_j) \\
	\label{last_brutal_g} = & \sum\limits_{u_i \neq 0}\sum\limits_{v_j \neq 0} |u_i v_j| - \left((1-t)\sum\limits_{u_i \neq 0} \mathrm{sign}(u_i)x_i+t\sum\limits_{u_i \neq 0}\mathrm{sign}(u_i)u_i \theta \right) \times \\ \label{last_brutal_h} & 
	\left((1-t)\sum\limits_{v_j \neq 0}  \mathrm{sign}(v_j)y_j+t\sum\limits_{v_j \neq 0}\mathrm{sign}(v_j)v_j/\theta \right) \\
	\label{last_brutal_i} = & \sum\limits_{u_i \neq 0}\sum\limits_{v_j \neq 0} |u_i v_j| - t^2 \sum\limits_{u_i \neq 0} \mathrm{sign}(u_i)u_i  \sum\limits_{v_j \neq 0} \mathrm{sign}(v_j)v_j \\
	\label{last_brutal_j} = & (1-t^2)f(x,y).
	\end{align}
\end{subequations}
Above, \eqref{last_brutal_a} is a consequence of the definition of $f$. By definition of $h$, we have $x_i+h_i = (1-t)x_i + tu_i$. According to \eqref{saddle_d}, $x_i=0$ whenever $u_i=0$, hence we also have that $x_i+h_i=0$ whenever $u_i=0$. Likewise $y_j+k_j=0$ whenever $v_j = 0$, and thus \eqref{last_brutal_b} holds. \eqref{last_brutal_c} is obtained by factorizing each term in the sum by $|u_i v_j|$. \eqref{eq:ration_preserved_1} implies that the term inside the absolute value is non-positive, hence \eqref{last_brutal_d}. \eqref{last_brutal_e} is the result of expanding the product inside the sum and the fact that $\mathrm{sign}(a) = |a|/a$ when $a \neq 0$. \eqref{last_brutal_f} is obtained by factorizing the second term in \eqref{last_brutal_e}. \eqref{last_brutal_g}-\eqref{last_brutal_h} uses the definition of $(h,k)$, namely that $h_i = tu_i\theta - tx_i$ and $k_j = tv_j/\theta - y_j$. \eqref{last_brutal_i} follows from the equalities in \eqref{saddle_c}, which results in two terms cancelling out. \eqref{last_brutal_j} is due to the fact that when $t=0$, the resulting expression must be equal to $f(x,y)$ since we are computing $f(x+t(u\theta-x),y+t(v/\theta-y))$. We conclude that $(x,y)$ is a saddle point and that it admits a direction of descent towards the global minimum $(u \theta, v/\theta)$.
\end{proof}

We can now deduce Theorem \ref{thm:landscape} from Proposition \ref{prop:classify}. Recall that we use the convention that a sum over an index set which is empty is equal to zero. For example, if none of the entries of $u$ are zero, then $\sum_{u_i=0} |x_i| = 0$. In this case, \eqref{eq:spurious_a} is not feasible since an absolute value cannot be negative. It follows from \eqref{eq:spurious_a}-\eqref{eq:spurious_b} that there are no spurious local minima if neither of the entries of $u$ nor $v$ are equal to zero. This condition on $u$ and $v$ is equivalent to saying the $uv^T$ has no zero entries. Conversely, if some entries of either $u$ or $v$ are equal to zero, then one can readily see that \eqref{eq:spurious_a}-\eqref{eq:spurious_b} is feasible. In the case where all the entries of $uv^T$ are equal to zero, it follows from the expression of $f(x,y)= \|x\|_1\|y\|_1$ that every local minimum is a global minimum. Theorem \ref{thm:landscape} naturally ensues.

\begin{figure}[ht]
    \centering
  \includegraphics[width=.9\linewidth]{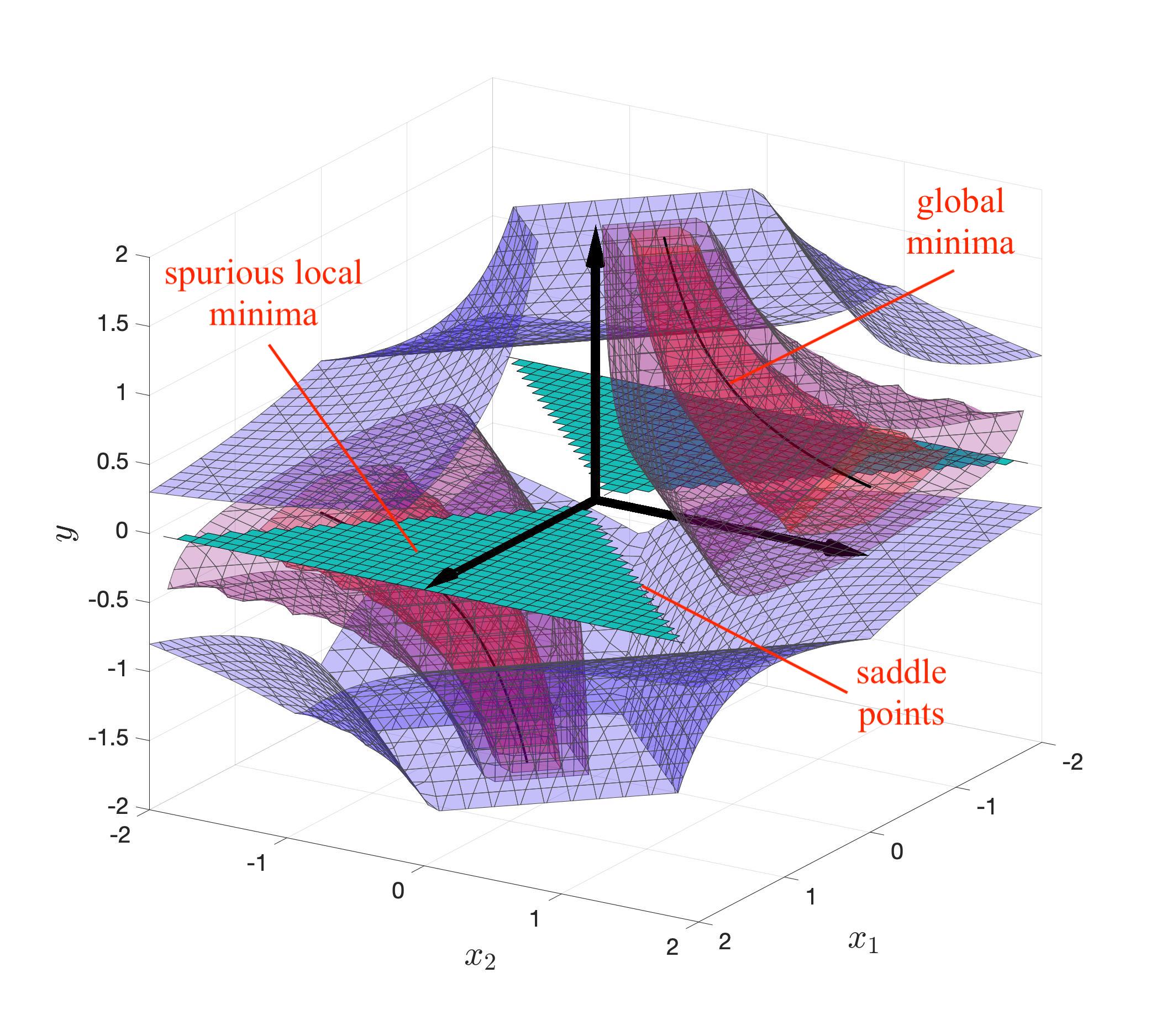}
  \caption{Landscape of $f(x_1,x_2,y):=|x_1y|+|x_2y-1|$.}
  \label{fig:contour}
\end{figure}

We finish this section with an example. When $u=(0,1)^T$ and $v = 1$, the landscape is represented in Figure \ref{fig:contour}.
According to \eqref{eq:global}, the global minima are all $(x_1,x_2,y) \in \mathbb{R}^3$ such that $(x_1,x_2,y) = (0,\theta,1/\theta)$ where $\theta \neq 0$. This corresponds to the two black hyperbolic branches. The surfaces around the global minima denote level sets of $f$, where the warmer the color (from blue to red), the smaller the objective value. Among the three level sets in the figure, the one with highest objective value (in blue) has a part which is not represented, namely, its intersection with the positive orthant. This is done in order to improve visibility. According to \eqref{eq:spurious_a}-\eqref{eq:spurious_b}, the spurious local minima are all $(x_1,x_2,y)\in \mathbb{R}^3$ such that $|x_2|<|x_1|$ and $y=0$. This corresponds to the area inside the two green triangles. According to \eqref{saddle_a}-\eqref{saddle_d}, the saddle points are all $(x_1,x_2,y)\in \mathbb{R}^3$ such that $|x_2|=|x_1|$ and $y=0$. This corresponds to the edges of the two green triangles located on the diagonal and anti-diagonal of the $(x_1,x_2)$-plane. Notice that the union of the global minima, the spurious local minima, and the saddle points in Figure \ref{fig:contour} agrees with the expression of the critical points found with a commercial solver in Figure \ref{fig:tarski}.

\end{document}